\documentclass[oneside, 11pt, reqno]{amsart}

\usepackage[utf8]{inputenc}

\usepackage{geometry}
\usepackage{mathtools}
\usepackage{enumitem}
\usepackage{amsmath}
\usepackage{amssymb}
\usepackage{amsthm}
\usepackage{etoolbox}
\usepackage{booktabs}
\usepackage[hidelinks]{hyperref}
\usepackage[capitalize]{cleveref}
\usepackage{tikz-cd}

\newtheorem{theorem}{Theorem}[section]
\newtheorem{lemma}[theorem]{Lemma}

\newtheorem{corollary}[theorem]{Corollary}

\theoremstyle{definition}
\newtheorem{definition}[theorem]{Definition}

\newtheorem{remark}[theorem]{Remark}

\newcommand{\upset}{\mathord\uparrow}
\newcommand{\downset}{\mathord\downarrow}

\keywords{Sober space; frame; Scott-open filter; compact saturated set; Hofmann-Mislove Theorem; Priestley duality}

\makeatletter
\@namedef{subjclassname@2020}{
  \textup{2020} Mathematics Subject Classification}
\patchcmd{\@setaddresses}{\indent}{\noindent}{}{}
\patchcmd{\@setaddresses}{\indent}{\noindent}{}{}
\patchcmd{\@setaddresses}{\indent}{\noindent}{}{}
\patchcmd{\@setaddresses}{\indent}{\noindent}{}{}
\makeatother
\author{Guram Bezhanishvili and Sebastian Melzer}
\address{\newline
Department of Mathematical Sciences\newline
New Mexico State University\newline
Las Cruces, NM 88003\newline
USA\newline}
\email{guram@nmsu.edu}
\email{smelzer@nmsu.edu}

\setlist[enumerate,1]{label={\upshape(\arabic*)}}
\title{Hofmann-Mislove through the Lenses of Priestley}
\subjclass{18F70; 06D22}

\begin{document}

\begin{abstract}
We use Priestley duality to give a new proof of the Hofmann-Mislove Theorem.
\end{abstract}

\maketitle

\section{Introduction}

Let $X$ be a sober space and $L=\mathcal O(X)$ the frame of open subsets of $X$. The Hofmann-Mislove Theorem \cite{HofmannMislove1981} establishes that the poset of Scott-open filters of $L$ is isomorphic to the poset of compact saturated subsets of $X$.  
This classic result was proved in 1981 and turned out to be an extremely useful link between topology and domain theory. Several alternative proofs of the theorem have been established since then (see, e.g., \cite{Compendium2003}). Of these, the proof by Keimel and Paseka \cite{KeimelPaseka1994} is probably the most direct and widely accepted.

There is a similar result in Priestley duality for distributive lattices \cite{Priestley1970, Priestley1972}, which establishes that the poset of filters of a bounded distributive lattice $L$ is isomorphic to the poset of closed upsets of the Priestley space $X$ of $L$. A close look at the two proofs reveals striking similarities. Indeed, it was pointed out in \cite[Rem.~6.4]{BezhanishviliGabelaiaKurz2010} that the two results are equivalent in the setting of spectral spaces 
(see Section~\ref{sec 3} for details). 

In this paper  we show that we can use Priestley duality to prove the Hofmann-Mislove Theorem. In fact, we will prove a more general result that the poset ${\sf OFilt}(L)$ of Scott-open filters of an arbitrary frame $L$ is isomorphic to the poset of compact saturated subsets of the space of points of $L$. This we do by establishing that ${\sf OFilt}(L)$ is isomorphic to the poset of the special closed upsets of the Priestley space of $L$, which we term Scott-upsets (see Section~\ref{sec 5} for details). The Hofmann-Mislove Theorem is an immediate consequence. 

We point out that the Hofmann-Mislove Theorem in this generality was proved in \cite[Thm.~8.2.5]{Vickers1989} using Zorn's lemma.
Since our new approach relies on Priestley duality, we only need to use the Prime Ideal Theorem, which is weaker than Zorn's lemma.

\section{Priestley duality}\label{sec 2}

Let $X$ be a poset. As usual, for $S\subseteq X$, we write
\begin{align*}
{\uparrow}S &= \{x\in X : s \le x \mbox{ for some } s\in S\},\\ 
{\downarrow}S &= \{x\in X : x \le s \mbox{ for some } s\in S\}.
\end{align*}
Then $S$ is an {\em upset} if $S={\uparrow}S$ and a {\em downset} if $S={\downarrow}S$. If $S=\{x\}$, we write ${\uparrow}x$ and ${\downarrow}x$ instead of ${\uparrow}S$ and ${\downarrow}S$. 

Let $X$ be a topological space. A subset $U$ of $X$ is {\em clopen} if it is both closed and open, $X$ is {\em zero-dimensional} if $X$ has a basis of clopen sets, and $X$ is a {\em Stone space} if $X$ is compact, Hausdorff, and zero-dimensional. 

\begin{definition}
A \emph{Priestley space} is a pair $(X,\le)$ where $X$ is a Stone space and $\leq$ is a partial order on $X$ satisfying the \emph{Priestley separation axiom}:
\begin{center}
    If $x \not \leq y$, then there is a clopen upset $U$ such that $x \in U$ and $y \not \in U$.
\end{center}
\end{definition}

When it is clear from the context, we simply write $X$ for a Priestley space. Let $\sf Pries$ be the category of Priestley spaces and continuous order-preserving maps. Let also $\sf Dist$ be the category of bounded distributive lattices and bounded lattice homomorphisms.

\begin{theorem} [Priestley duality \cite{Priestley1970,Priestley1972}] 
$\sf Pries$ is dually equivalent to $\sf Dist$.
\end{theorem}
 
We recall that for $D\in{\sf Dist}$, the Priestley space of $D$ is the set $X$ of prime filters of $D$ ordered by inclusion and topologized by the subbasis 
\[
\{ \sigma(a) : a \in D \} \cup \{ \sigma(b)^c : b \in D \},
\] 
where $\sigma:D \to \wp(X)$ is the Stone map $\sigma(a) = \{ x\in X : a \in x \}$ and $\wp(X)$ is the powerset of $X$. 
 
Each Priestley space comes equipped with two additional topologies: the topology $\tau_u$ of open upsets and the topology $\tau_d$ of open downsets. It is well known that clopen upsets form a basis for $\tau_u$. Since clopen upsets are exactly the compact opens of $\tau_u$, it follows that $(X,\tau_u)$ is a {\em coherent space} (the compact opens form a basis that is a bounded sublattice of the opens). Similarly, $(X,\tau_d)$ is a coherent space.

In addition, principal downsets ${\downarrow}x$ are exactly the join-irreducible elements in the lattice of closed downsets. Since closed downsets are the closed sets in $(X,\tau_u)$ and ${\downarrow}x$ is the closure of $\{x\}$ in $(X,\tau_u)$, we obtain that $(X,\tau_u)$ is a {\em sober space} (each closed irreducible set is the closure of a unique point). 

\begin{definition}
A topological space is {\em spectral} if it is sober and coherent. 
\end{definition}

Consequently, both $(X,\tau_u)$ and $(X,\tau_d)$ are spectral spaces. We call a map $f:X\to Y$ between spectral spaces a {\em spectral map} if the inverse image of each compact open in $Y$ is compact open in $X$. Let $\sf Spec$ be the category of spectral spaces and spectral maps. The assignment $X \mapsto (X,\tau_u)$ and $f \mapsto f$ defines a covariant functor from $\sf Pries$ to $\sf Spec$ which establishes that the two categories are isomorphic. 

\begin{theorem} [Cornish \cite{Cornish1975}] \label{thm: Cornish}
$\sf Pries$ is isomorphic to $\sf Spec$.
\end{theorem}  

Under this isomorphism, closed upsets of a Priestley space $X$ are exactly the compact saturated subsets of the spectral space $(X,\tau_u)$ (see, e.g., \cite[Thm.~6.1]{BezhanishviliGabelaiaKurz2010}), where we recall that saturated subsets of $(X,\tau_u)$ are exactly the upsets of $X$ (see also Section~\ref{sec 3}).

We conclude this section with the following well-known result in Priestley duality, which is reminiscent of the Hofmann-Mislove Theorem, and indeed will play a crucial role in our alternative proof of the theorem. 

Let $D$ be a distributive lattice and $X$ its Priestley space. Let ${\sf Filt}(D)$ be the poset of filters of $D$ ordered by reverse inclusion. Let also ${\sf ClUp}(X)$ be the poset of closed upsets of $X$ ordered by inclusion. We then have (see, e.g., \cite[Cor.~6.3]{BezhanishviliGabelaiaKurz2010}):

\begin{theorem}\label{thm: filters}
${\sf Filt}(D)$ is isomorphic to ${\sf ClUp}(X)$.
\end{theorem}

The isomorphism is obtained by sending $F\in{\sf Filt}(D)$ to the closed upset 
\[
K_F = \bigcap\{ \sigma(a) : a \in F \}.
\] 
Its inverse sends $K\in{\sf ClUp}(X)$ to the filter 
\[
F_K = \{ a \in D : K \subseteq \sigma(a) \}.
\]

\section{Hofmann-Mislove for spectral spaces}\label{sec 3}

We recall (see, e.g., \cite[p.~10]{PicadoPultr2012}) that a {\em frame} is a complete lattice $L$ satisfying
\[
a\wedge\bigvee S = \bigvee \{ a\wedge s : s \in S \}.
\]
A map $h:L\to M$ between two frames is a {\em frame homomorphism} if $h$ preserves finite meets and arbitrary joins. Let $\sf Frm$ be the category of frames and frame homomorphisms. 
The next definition is well known (see, e.g., \cite{KeimelPaseka1994}).

\begin{definition}
Let $L$ be a frame. A filter $F$ of $L$ is {\em Scott-open} if $\bigvee S \in F$ implies $\bigvee T \in F$ for some finite $T \subseteq S$. 
\end{definition}

Let $X$ be a sober space. We recall that the {\em specialization order} on $X$ is defined by $x\le y$ if $x$ belongs to the closure of $\{y\}$. A subset $S$ of $X$ is {\em saturated} if it is an upset in the specialization order. Let ${\sf KSat}(X)$ be the poset of compact saturated subsets of $X$ ordered by inclusion. Let also ${\sf OFilt}(L)$ be the poset of Scott-open filters of $L$ ordered by reverse inclusion. 

\begin{remark}
It is more customary to order ${\sf OFilt}(L)$ by inclusion and ${\sf KSat}(X)$ by reverse inclusion (see, e.g., \cite[Sec.~II-1]{Compendium2003}). Our ordering is motivated by how we ordered the posets of filters and closed upsets in Section~\ref{sec 2}.
\end{remark}

\begin{theorem} [Hofmann-Mislove \cite{HofmannMislove1981}]
Let $X$ be a sober space and $L$ the frame of open subsets of $X$. Then ${\sf OFilt}(L)$ is isomorphic to ${\sf KSat}(X)$.
\end{theorem}

As was pointed out in \cite[Rem.~6.4]{BezhanishviliGabelaiaKurz2010}, if $X$ is a spectral space, then the Hofmann-Mislove Theorem and Theorem~\ref{thm: filters} are simply reformulations of each other. Indeed, let $D$ be a bounded distributive lattice and let $L$ be the frame of ideals of $D$. Then $L$ is a coherent frame \cite[p.~64]{Johnstone1982}, where we recall that a frame $L$ is {\em coherent} if the compact elements form a bounded sublattice of $L$ that join-generates $L$.\footnote{We recall that $a\in L$ is {\em compact} if $a\leq\bigvee S$ implies $a\leq\bigvee T$ for some finite $T\subseteq S$, and $D$ {\em join-generates} $L$ if each element of $L$ is a join of elements from $D$.} In fact, sending $D$ to $L$ defines a covariant functor that establishes an equivalence between $\sf Dist$ and the category $\sf CohFrm$ of coherent frames and coherent morphisms (where a morphism is coherent if it is a frame homomorphism that sends compact elements to compact elements). Under this equivalence, the posets ${\sf Filt}(D)$ and ${\sf OFilt}(L)$ are isomorphic. 

Let $X$ be the Priestley space of $D$. By Theorem~\ref{thm: filters}, ${\sf Filt}(D)$ is isomorphic to ${\sf ClUp}(X)$. As we pointed out after Theorem~\ref{thm: Cornish}, ${\sf ClUp}(X)={\sf KSat}(X,\tau_u)$. Thus, the isomorphism of Theorem~\ref{thm: filters} between ${\sf Filt}(D)$ and ${\sf ClUp}(X)$ amounts to the Hofmann-Mislove isomorphism between ${\sf OFilt}(L)$ and ${\sf KSat}(X,\tau_u)$. Since $(X,\tau_u)$ is a spectral space and each spectral space arises this way (up to homeomorphism), we conclude that the Hofmann-Mislove Theorem for spectral spaces is equivalent to Theorem~\ref{thm: filters}.

In what follows we will show that we can use Priestley duality to prove the Hofmann-Mislove Theorem for an arbitrary sober space, and even more generally for an arbitrary frame. For this we will work with Priestley spaces of frames.

\section{Priestley duality for frames}

Let $L$ be a frame and $X$ its Priestley space. Since frames are nothing more than complete Heyting algebras, we can use Esakia duality \cite{Esakia1974} (see also \cite{Esakia2019}) to describe $X$. Indeed, the dual of a Heyting algebra is a Priestley space that in addition satisfies ${\downarrow}U$ is clopen for each clopen $U$. Such Priestley spaces are called {\em Esakia spaces}. Therefore, if $X$ is the Priestley space of a frame $L$, then $X$ is an Esakia space. In addition, since $L$ is complete, the closure of every open upset of $X$ is open in $X$ (see, e.g., \cite[Thm.~2.4]{BezhanishviliBezhanishvili2008}). Such spaces are called {\em extremally order-disconnected Esakia spaces}. Thus, Priestley spaces of frames are exactly the extremally order-disconnected Esakia spaces. Since frames are also referred to as locales, we adapt the following terminology.

\begin{definition}
    A \emph{localic space} or simply an {\em L-space} is an extremally order-disconnected Esakia space.
\end{definition}

\begin{remark}
In \cite{PultrSichler1988} these spaces are called f-spaces, and in \cite{PultrSichler2000} they are called LP-spaces.
\end{remark}

We recall (see, e.g., \cite[p.~15]{PicadoPultr2012}) that with each frame $L$ we can associate the space of points of $L$, where a {\em point} of a frame $L$ is a {\em completely prime filter} of $L$; that is, a point is a filter $F$ of $L$ such that $\bigvee S\in F$ implies $S\cap F \ne \varnothing$. Since each completely prime filter is prime, we will view the set $Y$ of points of $L$ as a subset of the Priestley space $X$ of $L$. By \cite[Lem.~5.1]{BezhanishviliGabelaiaJibladze2016} (see also \cite[Prop.~2.9]{PultrSichler2000}), we have: 

\begin{lemma} \label{lem: Y}
Let $L$ be a frame and $X$ the Priestley space of $L$. Then 
\[
Y = \{x \in X : \downset x \text{ is clopen}\}.
\]
\end{lemma}

We thus can define $Y$ in an arbitrary L-space.

\begin{definition}
For an L-space $X$, let $Y = \{x \in X : \downset x \text{ is clopen}\}$. 
\end{definition}

\begin{remark} \label{remark:4.5}
Since in a Priestley space the downset of a closed set is closed (see, e.g., \cite[Prop.~2.6]{Priestley1984}), we have $Y = \{x \in X : \downset x \text{ is open}\}$.
\end{remark}

For a frame $L$, let $\zeta(a)=\{y \in Y : a \in y \}$. It is well known (see, e.g., \cite[p.~15]{PicadoPultr2012}) that $\{ \zeta(a) : a \in L \}$ is a topology on $Y$. In fact, $\zeta(a)=\sigma(a)\cap Y$. Thus, the topology on $Y$ is the restriction of the open upset topology on $X$ (see \cite[Lem.~5.3]{AvilaBezhanishviliMorandiZaldivar2020}). 

We recall that a frame $L$ is {\em spatial} if $a \not\leq b$ implies that there is a completely prime filter $F$ of $L$ such that $a \in F$ and $b \notin F$. Equivalently, $L$ is spatial iff $L$ is isomorphic to the frame of open subsets of $Y$ (see, e.g., \cite[p.~18]{PicadoPultr2012}). By \cite[Thm.~5.5]{AvilaBezhanishviliMorandiZaldivar2020}, we have:

\begin{theorem} 
Let $L$ be a frame, $X$ its Priestley space, and $Y \subseteq X$ the set of points of $L$. Then
$L$ is spatial iff $Y$ is dense in $X$.
\end{theorem} 

We conclude this section with the following well-known fact (see, e.g., \cite[Lem.~2.3]{BezhanishviliBezhanishvili2008}), which will be used in the next section. As usual, we write $\sf cl$ for closure in a topological space.

\begin{lemma} \label{lem: join}
    Let $L$ be a frame and $X$ its Priestley space. For each $S\subseteq L$, we have
    \[
    \sigma\left(\bigvee S\right) = {\sf cl}\left(\bigcup_{s \in S} \sigma(s)\right).
    \]
\end{lemma}
\section{Hofmann-Mislove in full generality} \label{sec 5}

The key for obtaining a new proof of the Hofmann-Mislove Theorem using Priestley duality is the characterization of Scott-open filters of a frame $L$ as special closed upsets of the Priestley space $X$ of $L$. 

Let $K$ be a closed upset of $X$. Since $K$ is closed, it is well known (see, e.g., \cite[Thm.~3.2.1]{Esakia2019}) that for each $x\in K$ there is a minimal point $m$ of $K$ such that $m\leq x$. Thus, if $\min K$ is the set of minimal points of $K$, then $K={\uparrow}\min K$.

Let $F$ be a filter of $L$. We recall from Theorem~\ref{thm: filters} that the corresponding closed upset is $K_F=\bigcap\{\sigma(a) : a \in F\}$. We will freely use the well-known fact that in a Priestley space, the downset of a closed set is closed (see Remark~\ref{remark:4.5}). 

\begin{lemma} \label{lemma:scottopenIFFcoinitial}
    Let $L$ be a frame, $X$ its Priestley space, and $Y \subseteq X$ the set of points of $L$. For a filter $F$ of $L$, the following are equivalent.
    \begin{enumerate}[ref=\thetheorem(\arabic*)]
        \item $F$ is Scott-open.
        \item $\min K_F \subseteq Y$.
        \item For each open upset $U$ of $X$, from $K_F \subseteq {\sf cl}(U)$ it follows that $K_F \subseteq U$. \label{pultr-sichler-condition}
    \end{enumerate}
\end{lemma}
\begin{proof}
    (1)$\Rightarrow$(2)
    Suppose there is $x \in \min K_F \setminus Y$.
    Then $\downset x$ is not open (see Lemma~\ref{lem: Y}), so $U_x := (\downset x)^c$ is not closed. Therefore, ${\sf cl}(U_x)\cap\downset x\ne\varnothing$. But since $X$ is an Esakia space and $U_x$ is an upset, ${\sf cl}(U_x)$ is an upset (see \cite[Thm.~3.1.2(VI)]{Esakia2019}). Thus, since $x$ is the maximum of $\downset x$, we have $x \in {\sf cl}(U_x)$.  
    Let $S = \{s \in L : \sigma(s) \subseteq U_x\}$. Because $\downset x$ is a closed downset, $U_x$ is an open upset. Therefore, $U_x = \bigcup_{s \in S}\sigma(s)$ (see Section~\ref{sec 2}).
    Thus, 
    \[
    x \in {\sf cl}(U_x) = {\sf cl}\left(\bigcup_{s \in S}\sigma(s)\right) = \sigma(\bigvee S)
    \] 
    by Lemma~\ref{lem: join}. Since $(\min K_F) \cap \downset x = \{x\}$ and $x \in \sigma(\bigvee S)$, we have $K_F = \upset \min K_F \subseteq \sigma(\bigvee S)$. Therefore, $\bigvee S \in F$. On the other hand, for each $s\in S$ we have $x\notin\sigma(s)$. Since for each finite $T\subseteq S$, we have $\sigma(\bigvee T)=\bigcup_{s\in T}\sigma(s)$, we conclude that $x \notin \sigma(\bigvee T)$. Thus, $K_F \not\subseteq \sigma(\bigvee T)$, and hence $\bigvee T \notin F$ for each finite $T\subseteq S$. 
    Consequently, $F$ is not Scott-open.

    (2)$\Rightarrow$(3) Suppose $U$ is an open upset of $X$ such that $K_F \subseteq {\sf cl}(U)$. 
    Since $U$ is an open upset, it is a union of clopen upsets, so $U = \bigcup_{s\in S} \sigma(s)$ for some $S \subseteq L$. 
    Therefore, $K_F \subseteq {\sf cl}\left(\bigcup_{s\in S} \sigma(s)\right)$. Let $y \in \min K_F$. Then $y \in Y$ by (2), so $\downset y$ is open. Thus, $\downset y \cap \bigcup_{s\in S} \sigma(s) \neq \varnothing$. This means that for each $y \in \min K_F$ there is some $s \in S$ such that $y \in \sigma(s)$. 
    Consequently, 
    \[
    K_F = \upset \min K_F \subseteq \bigcup_{s\in S} \sigma(s) = U.
    \]
    
    (3)$\Rightarrow$(1) Suppose $\bigvee S \in F$ for some $S \subseteq L$. Then 
    \[
    K_F \subseteq \sigma(\bigvee S) = {\sf cl}\left({\bigcup_{s \in S} \sigma(s)}\right).
    \] 
    Therefore, $K_F \subseteq \bigcup_{s\in S} \sigma(s)$ by (3). Since $K_F$ is closed, it is compact, so there is a finite $T\subseteq S$ such that $K_F\subseteq \bigcup_{s\in T} \sigma(s)=\sigma(\bigvee T)$. Thus, $\bigvee T\in F$, and hence $F$ is Scott-open.
\end{proof}

\begin{definition}
    We call a closed upset $K$ of an L-space $X$ a \emph{Scott-upset} or {\em S-upset} for short if $\min K \subseteq Y$.
\end{definition}

\begin{remark}
In \cite{PultrSichler2000} closed sets satisfying Condition (3) of Lemma~\ref{lemma:scottopenIFFcoinitial} are called \emph{L-compact sets}. Thus, Scott-upsets are exactly the upsets of L-compact sets.
\end{remark}

\begin{corollary} \label{cor:compactness}
Let $L$ be a frame and $X$ its Priestley space.
    \begin{enumerate}[ref=\thetheorem(\arabic*)]
        \item $a \in L$ is compact iff $\sigma(a)$ is an S-upset.
        \item $L$ is compact iff $\min X \subseteq Y$. \label{cor:compactness-item}
    \end{enumerate}
\end{corollary}
\begin{proof}
    (1) Observe that $a$ is compact iff $\upset a$ is Scott-open. Now apply Lemma~\ref{lemma:scottopenIFFcoinitial}.
    
    (2) Since $\sigma(1)=X$, by (1) we have
    \begin{center}
    $L$  is compact $\Longleftrightarrow$ $1$ is compact $\Longleftrightarrow$ $\sigma(1)$  is an S-upset $\Longleftrightarrow$ $\min X \subseteq Y$. \qedhere
     \end{center}
\end{proof}

\begin{remark}
While Corollary~\ref{cor:compactness-item} is already known (see \cite[Lem.~3.1]{BezhanishviliGabelaiaJibladze2016}), our proof is particularly short. 
\end{remark}

For an L-space $X$, let ${\sf SUp}(X)$ be the subposet of ${\sf ClUp}(X)$ consisting of S-upsets of $X$. 

\begin{theorem} \label{thm:coinitialisscott}
    Let $L$ be a frame and $X$ its Priestley space. Then ${\sf OFilt}(L)$ is isomorphic to ${\sf SUp}(X)$.
\end{theorem}

\begin{proof}
    By Theorem~\ref{thm: filters}, ${\sf Filt}(L)$ is isomorphic to ${\sf ClUp}(X)$. By Lemma~\ref{lemma:scottopenIFFcoinitial}, this isomorphism restricts to an isomorphism between ${\sf OFilt}(L)$ and ${\sf SUp}(X)$.
\end{proof}

\begin{theorem} \label{thm:coinitialIFFcompactsaturated}
    Let $X$ be an L-space. Then ${\sf SUp}(X)$ is isomorphic to ${\sf KSat}(Y)$.
\end{theorem}

\begin{proof}
Define $f:{\sf SUp}(X) \to {\sf KSat}(Y)$ by $f(K)=K \cap Y$. We show that $f$ is well defined. 
By \cite[Lem.~5.3]{AvilaBezhanishviliMorandiZaldivar2020}, 
the specialization order on $Y$ is the restriction of the partial order on $X$ to $Y$. Therefore, since $K$ is an upset in $X$, we have that $K\cap Y$ is saturated in $Y$. To see that it is compact, let 
    $K \cap Y \subseteq \bigcup \zeta(a_i)$. We have
    \[
     \bigcup \zeta(a_i) = \bigcup (Y \cap \sigma(a_i)) = Y \cap \bigcup \sigma(a_i).
     \]
     Therefore, $K \cap Y \subseteq \bigcup \sigma(a_i)$. Since $K$ is an S-upset, $\min K\subseteq Y$, so $\min K\subseteq \bigcup \sigma(a_i)$, and hence $K=\upset\min K\subseteq \bigcup \sigma(a_i)$. Thus, because $K$ is compact in $X$, there are $a_{i_1},\dots,a_{i_n}$ such that $K\subseteq \sigma(a_{i_1})\cup\cdots\cup\sigma(a_{i_n})$. Consequently, $K\cap Y\subseteq\zeta(a_{i_1})\cup\cdots\cup\zeta(a_{i_n})$, implying that $K\cap Y$ is compact in $Y$. This yields that $f$ is well defined, and it clearly preserves $\subseteq$. 

Next define $g:{\sf KSat}(Y) \to {\sf SUp}(X)$ by $g(Q)=\upset Q$. 
    To see that $\upset Q$ is a closed upset, let $x \not \in \upset Q$. Then $y \not \leq x$ for all $y \in Q$.
    By the Priestley separation axiom, for each $y\in Q$ there is a clopen upset $U_y$ of $X$ such that $y \in U_y$ and $x \not \in U_y$.
    Therefore, $Q \subseteq \bigcup_{y \in Q} U_y$. Since $U_y\cap Y$ is open in $Y$, $Q$ is compact in $Y$, and a finite union of clopen upsets of $X$ is a clopen upset of $X$, we can conclude that there is a clopen upset $U$ of $X$ such that $Q \subseteq U$ and $x \not \in U$. Since $U$ is an upset of $X$, we also have $\upset Q \subseteq U$. Thus, $\upset Q$ is the intersection of clopen upsets of $X$ containing $\upset Q$, and hence  $\upset Q$ is a closed upset. Consequently, $g$ is well defined, and it clearly preserves $\subseteq$. 

    Finally, if $K$ is an S-upset of $X$, then $gf(K)=\upset (K \cap Y) = \upset \min K = K$; and if $Q$ is compact saturated in $Y$, then $fg(Q) = \upset Q \cap Y = Q$. Thus, $f$ and $g$ are order-preserving maps that are inverses of each other, and hence ${\sf SUp}(X)$ is isomorphic to ${\sf KSat}(Y)$.
\end{proof}

The Hofmann-Mislove Theorem is now an immediate consequence of Theorems~\ref{thm:coinitialisscott} and~\ref{thm:coinitialIFFcompactsaturated}. But we have proven the stronger result that the Scott-open filters of an arbitrary frame $L$ are isomorphic to the compact saturated sets of the space of points of $L$. The Hofmann-Mislove Theorem in this generality is established in \cite[Thm.~8.2.5]{Vickers1989}. The proof uses Zorn's Lemma, while our proof only relies on the Prime Ideal Theorem, which is a weaker principle.

\begin{corollary}[Hofmann-Mislove]$  $
    \begin{enumerate}
        \item Let $L$ be a frame and $Y$ the space of points of $L$. Then ${\sf OFilt}(L)$ is isomorphic to ${\sf KSat}(Y)$.
        \item If $Y$ is a sober space and $L$ is the frame of open subsets of $Y$, then ${\sf OFilt}(L)$ is isomorphic to ${\sf KSat}(Y)$. 
    \end{enumerate}
\end{corollary}

\begin{proof}
(1) Apply Theorems~\ref{thm:coinitialisscott} and~\ref{thm:coinitialIFFcompactsaturated}.

(2) Since $Y$ is sober, $Y$ is homeomorphic to the space of points of $L$ (see, e.g., \cite[p.~20]{PicadoPultr2012}). Now apply (1).  
\end{proof}

We conclude the paper with the following worthwhile consequence. We note that Corollary~\ref{cor:Scott-open-filter-is-intersection-of-completely-prime-filters} is proved in \cite[Lem.~8.2.2]{Vickers1989} using Zorn's lemma.

\begin{corollary}
Let $L$ be a frame and $X$ its Priestley space.
    \begin{enumerate}[ref=\thetheorem(\arabic*)]
        \item A Scott-open filter $F$ is completely prime iff $\min K_F$ is a singleton.
        \item Every Scott-open filter of $L$ is an intersection of completely prime filters of $L$. \label{cor:Scott-open-filter-is-intersection-of-completely-prime-filters}
    \end{enumerate}
\end{corollary}

\begin{proof}
(1) It is well known and easy to see that a Scott-open filter is completely prime iff it is prime. Thus, $F$ is completely prime iff $\min K_F$ is a singleton (see, e.g., \cite[Cor.~6.7]{BezhanishviliGabelaiaKurz2010}).

(2) Let $F$ be Scott-open and $a\notin F$. Then $K_F\not\subseteq\sigma(a)$. Therefore, there is $y\in\min K_F$ with $y\notin\sigma(a)$. 
By Lemma~\ref{lemma:scottopenIFFcoinitial}, $y\in Y$, so $y$ is completely prime. Moreover, $F\subseteq y$ because $y \in K_F$. Furthermore, $a\notin y$ since $y \not\in \sigma(a)$. Thus, there is a completely prime filter $y$ containing $F$ and missing $a$. Consequently, $F$ is the intersection of completely prime filters of $L$ containing~$F$.
\end{proof}

\bibliographystyle{abbrv}

\end{document}